\documentclass[11pt,a4paper,english]{amsart}

\usepackage[T1]{fontenc}
\usepackage[utf8]{inputenc}
\usepackage{amsmath}
\usepackage{amssymb}
\usepackage{amsthm}
\usepackage{latexsym}
\usepackage{amsfonts}

\usepackage[onehalfspacing]{setspace} 
\setdisplayskipstretch{.421} 
\newlength{\abovebis} 
\newlength{\belowbis} 
\newlength{\aboveshortbis} 
\newlength{\belowshortbis} 
\everydisplay\expandafter{%
  \the\everydisplay 
  \advance\abovedisplayskip\abovebis 
  \advance\belowdisplayskip\belowbis 
  \advance\abovedisplayshortskip\aboveshortbis 
  \advance\belowdisplayshortskip\belowshortbis 
}

\allowdisplaybreaks[3]

\def\R{\mathbb{R}}

\def\N{\mathbb{N}}
\def\C{\mathbb{C}}

\def\Ree{\mathrm{Re}}
\def\Imm{\mathrm{Im}}

\def\supp{\mathrm{supp}\,}
\def\E{\mathcal{E}}
\def\SE{\Sigma_E}

\theoremstyle{plain}

\newtheorem{lem}{Lemma}[section]
\newtheorem{theo}[lem]{Theorem}

\newtheorem{prop}[lem]{Proposition}

\theoremstyle{definition}

\newtheorem{rem}{Remark}[section]

\numberwithin{equation}{section}

\begin{document}
\title[Stability estimates in 2D at negative energy]{Stability estimates for an inverse problem for the Schr\"odinger equation at negative energy in two dimensions}
\author{Matteo Santacesaria}
\address[M. Santacesaria]{Centre de Mathématiques Appliquées -- UMR 7641, \'Ecole Polytechnique, 91128, Palaiseau, France}
\email{santacesaria@cmap.polytechnique.fr}
\begin{abstract}
We study the inverse problem of determining a real-valued potential in the two-dimensional Schr\"odinger equation at negative energy from the Dirichlet-to-Neumann map. It is known that the problem is ill-posed and a stability estimate of logarithmic type holds. In this paper we prove three new stability estimates. The main feature of the first one is that the stability increases exponentially with respect to the smoothness of the potential, in a sense to be made precise. The others show how the first estimate depends on the energy, for low and high energies (in modulus). In particular it is found that for high energies the stability estimate changes, in some sense, from logarithmic type to Lipschitz type: in this sense the ill-posedness of the problem decreases when increasing the energy (in modulus).
\end{abstract}

\maketitle

\section{Introduction}
The problem of the recovery of a potential in the Schr\"odinger equation from boundary measurements, the Dirichlet-to-Neumann map, has been studied since the 1980s, namely in connection with Calder\'on's inverse conductivity problem. The aim of this paper is to give new insights about its stability issues.

It is well known that the problem is ill-posed: Alessandrini \cite{A} proved that a logarithmic stability holds and Mandache \cite{M} showed that it was optimal, in some sense. Nevertheless, Mandache's result provided also the information that stability could be increased in a way depending on the smoothness of potentials. Optimal stability estimates, with respect to smoothness of potentials, were indeed recently obtained in \cite{N2} and \cite{S} in dimensions $d \geq 3$ and $d=2$, respectively (at zero energy). However, even for smooth potentials the problem remains ill-posed.

It was observed that one way to increase stability is to modify another factor in the equation: the energy. Indeed, at high energies the ill-posedness diminishes considerably: this motivated some rapidly converging approximation algorithms in two and three dimensions \cite{N4}, \cite{N6}, \cite{NS2} and stability estimates of Lipschitz-logarithmic type explicitly depending on the energy in three dimensions \cite{Isa}.

In this paper we continue the work started in \cite{S}, at zero energy (the Calder\'on problem), and give new stability estimates depending on the smoothness of potentials and the energy. We restricted ourself to the negative energy case, for the simplicity of the proofs. Results for the positive energy case are indeed similar in many respects and will be published in a subsequent paper.\smallskip

We consider the Schr\"odinger equation at fixed energy $E$,
\begin{equation} \label{equa}
(-\Delta + v)\psi = E \psi \quad \text{on } D, \quad E \in \R,
\end{equation}
where $D$ is a open bounded domain in $\R^2$ and $v \in L^{\infty}(D)$ (we will refer to $v$ as a \textit{potential}).
Under the assumption that
\begin{equation} \label{direig}
0 \textrm{ is not a Dirichlet eigenvalue for the operator } - \Delta + v -E\textrm{ in } D,
\end{equation}
we can define the Dirichlet-to-Neumann operator $\Phi(E) :H^{1/2}(\partial D) \to H^{-1/2}(\partial D)$, corresponing to the potential $v$, as follows:
\begin{equation} \label{defdtn1}
\Phi(E)f = \left. \frac{\partial u}{\partial \nu}\right|_{\partial D},
\end{equation}
where $f \in H^{1/2}(\partial D)$, $\nu$ is the outer normal of $\partial D$, and $u$ is the $H^1(D)$-solution of the Dirichlet problem
\begin{equation}
(-\Delta + v)u = Eu \; \textrm{on} \; D, \; \; \; u|_{\partial D}=f.
\label{schr}
\end{equation}
The following inverse problem arises from this construction.

\textbf{Problem 1.} Given $\Phi(E)$ for a fixed $E \in \R$, find $v$ on $D$.

This problem can be considered as the Gel’fand inverse boundary value problem for the two-dimensional Schr\"odinger equation at fixed energy (see \cite{G}, \cite{N1}). At zero energy this problem can be seen also as a generalization of the Calder\'on problem of the electrical impedance tomography (see \cite{C}, \cite{N1}).
In addition, the history of inverse problems for the two-dimensional Schr\"odinger equation at fixed energy goes back to \cite{DKN} (see also \cite{N3, Gr} and reference therein). Problem 1 can also be considered as an example of ill-posed problem: see \cite{LRS}, \cite{BK} for an introduction to this theory. 

Note that this problem is not overdetermined, in the sense that we consider the reconstruction of a function $v$ of two variables from inverse problem data dependent on two variables.\smallskip

In this paper we study interior stability estimates, i.e. we want to prove that given two Dirichlet-to-Neumann operators $\Phi_1(E)$ and $\Phi_2(E)$, corresponding to potentials $v_1$ and $v_2$ on $D$, we have that
\begin{equation} \nonumber
\|v_1 - v_2\|_{L^{\infty}(D)} \leq \omega \left( \| \Phi_1(E) - \Phi_2(E)\|_{H^{1/2}(\partial D) \to H^{-1/2}(\partial D)}\right),
\end{equation}
where the function $\omega(t) \to 0$ as fast as possible as $t \to 0$ at any fixed $E$. The explicit dependence of $\omega$ on $E$ is analysed as well.

There is a wide literature on the Gel'fand inverse problem at fixed energy (i.e. Problem 1 in multidimensions). In the case of complex-valued potentials the global injectivity of the map $v \to \Phi$ was firstly proved in \cite{N1} for $D \subset \R^d$ with $d \geq 3$ and in \cite{B} for $d = 2$ with $v \in L^p$: in particular, these results were obtained by the use of global reconstructions developed in the same papers. A global stability estimate for Problem 1 for $d \geq 3$ was first found by Alessandrini in \cite{A}; a principal improvement of this result was given recently in \cite{N2}. In the two-dimensional case the first global stability estimate was given in \cite{NS}.
Note that for the Calder\'on problem (of the electrical impedance tomography) in its initial formulation the global uniqueness was firstly proved in \cite{SU} for $d \geq 3$ and in \cite{Na2} for $d = 2$. In addition, for the case of piecewise constant or piecewise real analytic conductivity the first uniqueness results for the Calder\'on problem in dimension $d \geq 2$ were given in \cite{D} and \cite{KV2}. In the case of piecewise constant conductivities a Lipschitz stability estimate was proved in \cite{AV} (see \cite{R} for additional studies in this direction).

Most stability results for the Calder\'on problem in two dimensions have been formulated with the goal of proving stability estimates using the least regular conductivities possible (see \cite{L}, \cite{BFR}). Instead, we have tried to address different questions: how the estimates vary with respect to the smoothness of the potentials and the energy.

The results, detailed below, constitute also a progress in the non-smooth case: they indicate stability dependence of the smooth part of a singular potential with respect to boundary value data.\smallskip

We will assume for simplicity that
\begin{equation} \label{cv1}
\begin{split}
&D \text{ is an open bounded domain in } \R^2, \qquad \partial D \in C^2, \\
&v \in W^{m,1}(\R^2) \text{ for some } m > 2, \quad \bar v = v, \quad \mathrm{supp} \; v \subset D,
\end{split}
\end{equation} 
where
\begin{align}
&W^{m,1}(\R^2) = \{ v \; : \; \partial^J v \in L^1(\R^2),\; |J| \leq m \}, \qquad m \in \N \cup \{0\},\\ \nonumber
&J \in (\N \cup \{0\})^2, \qquad |J| = J_1+J_2, \qquad \partial^J v(x) = \frac{\partial^{|J|} v(x)}{\partial x_1^{J_1} \partial x_2^{J_2}}.
\end{align}
Let
\begin{equation} \nonumber
\|v\|_{m,1} = \max_{|J| \leq m} \| \partial^J v \|_{L^1(\R^2)}.
\end{equation}
We will need the following regularity condition:
\begin{equation} \label{cv2}
|E| > E_1, 
\end{equation}
where $E_1 = E_1(\|v\|_{m,1},D)$. This condition implies, in particular, that the Faddeev eigenfunctions are well-defined on the entire fixed-energy surface in the spectral parameter.

\begin{theo} \label{maintheo}
Let the conditions \eqref{direig}, \eqref{cv1}, \eqref{cv2} hold for the potentials $v_1, v_2$, where $D$ is fixed, and let $\Phi_1(E)$ , $\Phi_2(E)$ be the corresponding Dirichlet-to-Neumann operators at fixed negative energy $E < 0$. Let $\|v_j\|_{m,1} \leq N$, $j=1,2$, for some $N >0$. Then there exists a constant $c_1 = c_1(E,D, N,m)$ such that
\begin{equation} \label{est1}
\| v_2 - v_1\|_{L^{\infty}(D)} \leq c_1(\log(3 + \| \Phi_2(E) - \Phi_1(E) \|_*^{-1} ))^{-\alpha},
\end{equation}
where $\alpha = m-2$ and $\| \Phi_2 - \Phi_1 \|_* = \| \Phi_2 - \Phi_1 \|_{H^{1/2}(\partial D) \to H^{-1/2}(\partial D)}$.

Moreover, there exists a constant $c_2 = c_2(D,N,m,p)$ such that for any $0 < \kappa < 1/(l+2)$, where $l = \mathrm{diam}( D)$, we have
\begin{align} \label{est2}
\|v_2 - v_1\|_{L^{\infty}(D)}&\leq c_2 \bigg[\left(|E|^{1/2}+\kappa \log(3+\delta^{-1})\right)^{-(m-2)}\\ \nonumber
&\quad +  \delta (3+\delta^{-1})^{\kappa(l+2)}e^{|E|^{1/2}(l+3)}\bigg],
\end{align}
where $\delta = \|\Phi_2(E)-\Phi_1(E)\|_*$.

In addition, there exists a constant $c_3 = c_3(D,N,m,p)$ such that for $E, \delta$ which satisfy
\begin{equation} \label{hh2}
|E|^{1/2} > \log(3+\delta^{-1}), \quad |E| > 1,
\end{equation}
we have
\begin{align} \label{est3}
\| v_2 - v_1\|_{L^{\infty}(D)} \leq c_3 \bigg[|E|^{-(m-2)/2}\log(3+\delta^{-1})^{-(m-2)} + \delta e^{|E|(l+3)}\bigg].
\end{align}
\end{theo}

The novelty of estimate \eqref{est1}, with respect to \cite{NS}, is that, as $m \to +\infty$, we have $\alpha \to +\infty$. Moreover, under the assumption of Theorem \ref{maintheo}, according to instability estimates of Mandache \cite{M} and Isaev \cite{I}, our result is almost optimal. To be more precise, it was proved that stability estimate \eqref{est1} cannot hold for $\alpha > 2m$ for real-valued potentials and $\alpha > m$ for complex-valued potentials. Indeed, our estimates are still valid for complex-valued potentials, if $|E|$ is sufficiently large with respect to $\|v\|_{C(\bar D)}$.

In addition, estimate \eqref{est1} extends the result obtained in \cite{S} for the same problem at zero energy. In dimension $d \geq 3$ a global stability estimate similar to \eqref{est1} was proved in \cite{N2}, at zero energy.\smallskip

As regards \eqref{est2} and \eqref{est3}, their main feature is the explicit dependence on the energy $E$. These estimates consist each one of two parts, the first logarithmic and the second H\"older or Lipschitz; when $|E|$ increases, the logarithmic part decreases and the H\"older/Lipschitz part becomes dominant. 

These estimates, namely \eqref{est3}, are coherent with the approximate reconstruction algorithm developed in \cite{N4} and \cite{NS2} at positive energy. In fact, inequalities like \eqref{est1}, \eqref{est2} and \eqref{est3} should be valid also for the Schr\"odinger equation at positive energy.

Note that, for Problem 1 in three dimensions, global energy-dependent stability estimates changing from logarithmic type to Lipschitz type for high energies were given recently in \cite{Isa}. However these estimates are given in the $L^2(D)$ norm and without any dependence on the smoothness of the potentials.\smallskip

The proof of Theorem \ref{maintheo} follows the scheme of \cite{S} and it is based on the same $\bar \partial$ techniques.
The map $\Phi(E) \to v(x)$ is considered as the composition of $\Phi(E) \to r(\lambda)$ and $r(\lambda) \to v(x)$, where $r(\lambda)$ is a complex valued function, closely related to the so-called generalised scattering amplitude (see Section \ref{sec2} for details).

The stability of $\Phi(E) \to r(\lambda)$ -- previously known only for $E =0$ -- relies on an identity of \cite{N5} (based in particular on \cite{A}), and estimates on $r(\lambda)$ for $\lambda$ near $0$ and $\infty$. The estimate is of logarithmic type, with respect to $\Phi$ (at fixed $E$): it is proved in section \ref{secphir}. Note that the results of this section are valid also for positive energy.

The stability of $r(\lambda) \to v(x)$ is of H\"older type and follows the same arguments as in \cite[Section 4]{S}. The composition of the two above-mentioned maps gives the result of Theorem \ref{maintheo}, as showed in Section 4.

\begin{rem}
We point out another possible approach to obtain inequality \eqref{est1}. The approach is based on the following observation (which follows from \cite[Basic Lemma]{GN}): for potentials $v$ satisfying the assumptions of Theorem \ref{maintheo} we have that $v-E$ is of conductivity type, i.e. there exists a positive real-valued function $\psi_0 \in L^{\infty}(D)$ bounded from below such that
\begin{equation}
v-E = \frac{\Delta \psi_0}{\psi_0}.
\end{equation}
Thus Problem 1 at fixed negative energy is reduced to the the same problem at zero energy for the conductivity-type potential $\frac{\Delta \psi_0}{\psi_0}$. It is then possible to apply the result of \cite{S} and find the same stability estimate.
\end{rem}
\begin{rem}
In a similar way as in \cite{IN}, the stability estimates of Theorem \ref{maintheo} can be extended to the case when we do not assume that condition \eqref{direig} is fulfilled and consider the Cauchy data set instead of the Dirichlet-to-Neumann map $\Phi(E)$.
\end{rem}

This work was fulfilled in the framework of research carried out under the supervision of R.G. Novikov.

\section{Preliminaries} \label{sec2}
We recall the definition of the Faddeev eigenfunctions $\psi(x,k)$ of equation \eqref{equa}, for $x=(x_1,x_2) \in \R^2$, $k=(k_1,k_2) \in \SE \subset \C^2$, $\SE = \{ k \in \C^2 : k^2 = k_1^2 + k_2^2 = E\}$ for $E \neq 0$ (see \cite{F}, \cite{N3}, \cite{Gr}). We first extend $v \equiv 0$ on $\R^2 \setminus D$ and define $\psi(x,k)$ as the solution of the following integral equation:
\begin{align} \label{inteq}
\psi(x,k) &= e^{ikx}+\int_{y \in \R^2} G(x-y, k)v(y)\psi(y,k) dy,\\ \label{greenfad}
G(x,k)&=g(x,k)e^{ikx}, \\
g(x,k)&=-\left(\frac{1}{2 \pi}\right)^2 \int_{\xi \in \R^2} \frac{e^{i\xi x}}{\xi^2+2k\xi}d\xi,
\end{align}
where $x \in \R^2$, $k\in \SE \setminus \R^2$. It is convenient to write \eqref{inteq} in the following form
\begin{equation} \label{inteq2}
\mu(x,k) = 1 + \int_{y \in \R^2} g(x-y,k)v(y)\mu(y,k)dy,
\end{equation}
where $\mu(x,k)e^{ikx} = \psi(x,k)$.

We define $\E_E \subset \SE \setminus \R^2$ the set of exceptional points of integral equation \eqref{inteq2}: $k \in \SE \setminus (\E_E \cup \R^2)$ if and only if equation \eqref{inteq2} is uniquely solvable in $L^{\infty}(\R^2)$.

\begin{rem} \label{reme}
From \cite[Proposition 1.1]{N4} we have that there exists $E_0 = E_0(\|v\|_{m,1},D)$ such that for $|E| \geq E_0(\|v\|_{m,1},D)$ there are no exceptional points for equation \eqref{inteq2}, i.e. $\E_E = \emptyset$: thus the Faddeev eigenfunctions exist (unique) for all $k \in \SE \setminus \R^2$.
\end{rem}

Following \cite{GN}, \cite{N3}, we make the following change of variables
\begin{equation} \nonumber
z = x_1 + i x_2, \qquad \lambda = \frac{k_1+ik_2}{\sqrt{E}}, 
\end{equation}
and write $\psi, \mu$ as functions of these new variables. For $k \in \SE \setminus (\E_E \cup \R^2)$ we can define, for the corresponding $\lambda$, the following generalised scattering amplitude:
\begin{align}
b(\lambda,E) &= \frac{1}{(2 \pi)^2} \! \! \int_{\C} \exp \bigg[\frac i 2 \sqrt{E}\left(1+(\mathrm{sgn}\,E)\frac{1}{\lambda \bar \lambda} \right)\\ \nonumber
&\quad \times \left( (\mathrm{sgn}\,E)z\bar \lambda + \lambda \bar z\right)\bigg]v(z)\mu(z,\lambda)d\Ree z\,d\Imm z.
\end{align}
This function plays an important role in the inverse problem because of the following $\bar \partial$-equation, which holds when $v$ is real-valued (see \cite{N3} for more details):
\begin{equation}\label{dbar}
\frac{\partial}{\partial \bar \lambda}\mu(z,\lambda) = r(z,\lambda)\overline{\mu(z,\lambda)},
\end{equation}
for $\lambda$ not an exceptional point (i.e. $k(\lambda) \in \SE \setminus (\E_E \cup \R^2)$), where
\begin{align} \label{defr2}
r(z,\lambda) &= r(\lambda)\exp \bigg[\frac i 2 \sqrt{E}\left(1+(\mathrm{sgn}\,E)\frac{1}{\lambda \bar \lambda} \right) \left( (\mathrm{sgn}\,E)z\bar \lambda + \lambda \bar z\right)\bigg],\\  \label{defr}
r(\lambda) &= \frac{\pi}{\bar \lambda}\mathrm{sgn}(\lambda \bar \lambda -1) b(\lambda,E).
\end{align}

We recall that if $v \in W^{m,1}(\R^2)$ with $\supp v \subset D$, then $\|\hat{v} \|_m < +\infty$, where
\begin{gather}
\hat v(p) = (2 \pi)^{-2} \int_{\R^2} e^{ipx} v(x) dx, \qquad p \in \C^2, \\
\|u\|_m = \sup_{p \in \R^2} | (1+|p|^2)^{m/2} u(p)|,
\end{gather}
for a test function $u$.

The following lemma is a variation of a result in \cite{N4}:
\begin{lem} \label{lemb}
Let the conditions \eqref{cv1}, \eqref{cv2} hold for a potentials $v$ and let $E \in \R \setminus \{0\}$. Then there exists an $R = R(m,\|\hat{v} \|_m) > 1$, such  that
\begin{equation}
|b(\lambda,E)| \leq 2 \|\hat v\|_m \left( 1+|E|\left(|\lambda|+\mathrm{sgn}(E)/|\lambda| \right)^2 \right)^{-m/2},
\end{equation}
for $|\lambda| > \frac{2R}{|E|^{1/2}}$ and $|\lambda|<\frac{|E|^{1/2}}{2R}$.
\end{lem}
\begin{proof}
We consider the function $H(k,p)$ defined as
\begin{equation}
H(k,p) = \frac{1}{(2 \pi)^2}\int_{\R^2}e^{i(p-k)x}v(x)\psi(x,k) dx,
\end{equation}
for
\begin{equation}
k = k(\lambda)= \left( \frac{\sqrt{E}}{2}(\lambda+\lambda^{-1}), \frac{i\sqrt{E}}{2}(\lambda^{-1}-\lambda)\right),
\end{equation}
$\lambda \in \C\setminus \{ 0 \}$, $\Imm\, k(\lambda) \neq 0$, $p \in \R^2$ and $\psi(x,k)$ as defined at the beginning of this section. Since $\E_E = \emptyset$ (see Remark \ref{reme}), the function $H(k(\lambda), k(\lambda) + \overline{k(\lambda)})=b(\lambda,E)$ is defined for every $\lambda \in \C \setminus \{0\}$. Then, by \cite[Proposition 1.1, Corollary 1.1]{N4} (see also Remark \ref{beh}) we have
\begin{equation} \label{esttt}
|H(k,p)| \leq 2 \|\hat v\|_m ( 1+p^2)^{-m/2}, \qquad \text{for } |k| > R(m,\|\hat{v} \|_m),
\end{equation}
where $|k|=(|\Ree k|^2 + |\Imm k|^2)^{1/2}$. This finishes the proof of Lemma \ref{lemb}.
\end{proof}

At several points in the paper we will use \cite[Lemma 2.1]{N4}, which we restate in an adapted form.
\begin{lem} \label{lem21}
Let the conditions \eqref{cv1}, \eqref{cv2} hold for a potentials $v$. Let $\mu(x,k)$ be the associated Faddeev functions. Then, for any $0<\sigma<1$, we have
\begin{align}
|\mu(x,k)-1|+\left|\frac{\partial \mu(x,k)}{\partial x_1}\right|+\left|\frac{\partial \mu(x,k)}{\partial x_2}\right| \leq |k|^{-\sigma}c(m,\sigma) \|\hat v\|_m, 
\end{align} 
for $k \in \C^2$ such that $k^2 < 0$ and $|k| \geq R$, where $R$ is defined in Lemma \ref{lemb}.
\end{lem}
Throughout all the paper $c(\alpha, \beta, \ldots)$ is a positive constant depending on parameters $\alpha, \beta, \ldots$

\begin{rem} \label{beh}
Even if \cite[Proposition 1.1, Corollary 1.1, Lemma 2.1]{N4} were proved for $E >0$, they are still valid in the negative energy case (and zero energy case).
\end{rem}

We also restate \cite[Lemma 2.6]{BBR}, which will be useful in section \ref{secrv}.
\begin{lem}[\cite{BBR}] \label{lemtech}
Let $q_1\in L^{s_1}(\C) \cap L^{s_2}(\C)$, $1 < s_1 <2 < s_2 < \infty$ and $q_2 \in L^s( \C)$, $1 < s <2$. Assume $u$ is a function in $L^{\tilde s}(\C)$, with $1/\tilde s = 1/s - 1/2$, which satisfies
\begin{equation}
\frac{\partial u (\lambda)}{\partial \bar \lambda} = q_1(\lambda) \bar u(\lambda) + q_2(\lambda), \qquad \lambda \in \C.
\end{equation}
Then there exists $c=c(s,s_1,s_2) >0$ such that
\begin{equation}
\|u\|_{L^{\tilde s}} \leq c \|q_2\|_{L^s} \exp(c (\|q_1\|_{L^{s_1}}+\|q_1\|_{L^{s_2}})).
\end{equation}
\end{lem}

We will make also use of the well-known H\"older's inequality, which we recall in a special case: for $f \in L^p(\C)$, $g \in L^q(\C)$ such that $1 \leq p,q\leq \infty$, $1\leq r < \infty$, $1/p+1/q = 1/r$, we have
\begin{equation} \label{holder}
 \|fg\|_{L^r(\C)}\leq \|f\|_{L^p(\C)}\|g\|_{L^q(\C)}.
\end{equation}

\section{From $\Phi(E)$ to $r(\lambda)$} \label{secphir}

\begin{lem} \label{lemestr}
Let the conditions \eqref{cv1}, \eqref{cv2} hold and take $0<a_1< \min\left(1, \frac{|E|^{1/2}}{2R}\right)$, $a_2 > \max\left(1,\frac{2R}{|E|^{1/2}}\right)$,
for $E\in \R \setminus \{0\}$ and $R$ as defined in Lemma \ref{lemb}. Then for $p \geq 1$ we have
\begin{align} \label{estlem21}
\left\||\lambda|^j r(\lambda)\right\|_{L^p(|\lambda|<a_1)} &\leq c(p,m) \|\hat v\|_m |E|^{-m/2} a_1^{m-1+j+2/p}, \\ \label{estlem22}
\left\||\lambda|^j r(\lambda)\right\|_{L^p(|\lambda|>a_2)} &\leq c(p,m) \|\hat v\|_m |E|^{-m/2} a_2^{-m-1+j+2/p},
\end{align}
where $j=1,0,-1$ and $r$ was defined in \eqref{defr}.
\end{lem}
\begin{proof}
It is a corollary of Lemma \ref{lemb}. Indeed $|r(\lambda)| = \pi |b(\lambda,E)|/|\lambda|$ and
\begin{align*}
\||\lambda|^j r \|_{L^p(|\lambda| < a_1)}^p &\leq c \left(\frac{\|\hat v\|_m}{|E|^{m/2}}\right)^p \int_{t < a_1} t^{1+(m-1+j)p} dt\\
&= c(p,m) \left(\frac{\|\hat v\|_m}{|E|^{m/2}}\right)^p a_1^{(m-1+j)p+2},
\end{align*}
which gives \eqref{estlem21}. The proof of \eqref{estlem22} is analogous.
\end{proof}

\begin{lem} \label{lemdifh}
Let $D \subset \{ x \in \R^2 \, : \, |x| \leq l\}$, $E < 0$, $v_1, v_2$ be two potentials satisfying \eqref{direig}, \eqref{cv1}, \eqref{cv2}, $\Phi_1(E), \Phi_2(E)$ the corresponding Dirichlet-to-Neumann operator and $b_1, b_2$ the corresponding generalised scattering amplitude. Let $\|v_j\|_{m,1} \leq N$, $j=1,2$. Then we have
\begin{equation} \label{estdifh}
|b_2(\lambda) - b_1(\lambda)|\leq c(D,N)e^{(l+1)\sqrt{|E|}(|\lambda|+1/|\lambda|)}\|\Phi_2(E) - \Phi_1(E)\|_*, \; \lambda \in \C \setminus \{0\}.
\end{equation}
\end{lem}
\begin{proof}
We have the following identity:
\begin{equation} \label{aless}
b_2(\lambda) - b_1(\lambda) = \left(\frac{1}{2\pi}\right)^2\int_{\partial D}\psi_1(x,\overline{k(\lambda)})(\Phi_2(E) - \Phi_1(E))\psi_2(x,k(\lambda)) dx,
\end{equation}
where $\psi_i(x,k)$ are the Faddeev functions associated to the potential $v_i$, $i=1,2$. This identity is a particular case of the one in \cite[Theorem 1]{N5}: we refer to that paper for a proof.

From this identity we obtain:
\begin{align} \label{estlem1}
|b_2(\lambda) - b_1(\lambda)| \leq \frac{1}{(2\pi)^2} \|\psi_1(\cdot,k)\|_{H^{1/2}(\partial D)}\|\Phi_2(E) - \Phi_1(E)\|_* \|\psi_2(\cdot,k)\|_{H^{1/2}(\partial D)}.
\end{align}
Now, for $\tilde p > 2$, using the trace theorem and Lemma \ref{lem21} we get
\begin{align*}
&\|\psi_j(\cdot,k(\lambda))\|_{H^{1/2}(\partial D)} \leq c \|\psi_j(\cdot,k(\lambda))\|_{W^{1,\tilde p}(D)}\\
&\quad \leq c\frac{\sqrt{|E|}}{2}l (|\lambda|+1/|\lambda|) e^{\frac{\sqrt{|E|}}{2}l (|\lambda|+1/|\lambda|)} \|\mu_j(\cdot,k(\lambda))\|_{W^{1,\tilde p}(D)}\\
&\quad \leq c\, e^{\frac{\sqrt{|E|}}{2}(l+1) (|\lambda|+1/|\lambda|)} \|\mu_j(\cdot,k(\lambda))\|_{W^{1,\tilde p}(D)} \leq c(D,N,m) e^{\frac{\sqrt{|E|}}{2}(l+1) (|\lambda|+1/|\lambda|)},
\end{align*}
for $j=1,2.$ This, combined with \eqref{estlem1}, gives \eqref{estdifh}.
\end{proof}

Now we turn to the main result of the section.
\begin{prop} \label{proprdir}
Let $E < 0$ be such that $|E|\geq  E_1 = \max ((2R)^2, E_0)$, where $R$ is defined in Lemma \ref{lemb} and $E_0$ in Remark \ref{reme}, let $v_1, v_2$ be two potentials satisfying \eqref{direig}, \eqref{cv1}, \eqref{cv2}, $\Phi_1(E), \Phi_2(E)$ the corresponding Dirichlet-to-Neumann operator and $r_1, r_2$ as defined in \eqref{defr}. Let $\|v_k\|_{m,1} \leq N$, $k=1,2$. Then for every $p \geq 1 $ there exists a constant $\theta_1= \theta_1(E,D, N,m,p)$ such that
\begin{equation} \label{mainesth}
\||\lambda|^j |r_2 - r_1|\|_{L^p(\C)} \leq \theta_1 \log(3 + \delta^{-1})^{-(m-2)},
\end{equation}
for $j=-1,0,1$, $\delta = \|\Phi_2(E) - \Phi_1(E)\|_*$. Moreover, there exists a constant $\theta_2 = \theta_2(D,N,m,p)$ such that for any $0 < \kappa < \frac{1}{l+2}$, where $l = \mathrm{diam}(D)$, and for $|E| \geq E_1$ we have
\begin{align} \label{mainesth2}
\||\lambda|^j |r_2 - r_1|\|_{L^p(\C)} &\leq \theta_2 \bigg[|E|^{-1} \left(|E|^{1/2}+\kappa \log(3+\delta^{-1})\right)^{-(m-2)} \\ \nonumber
&\quad + \frac{\delta(3+\delta^{-1})^{\kappa(l+2)}}{|E|^{1/2p}}e^{|E|^{1/2}(l+2)}\bigg], \quad j=-1,0,1.
\end{align}
In addition, there exists a constant $\theta_3 = \theta_3(D,N,m,p)$ such that for $E, \delta$ which satisfy
\begin{equation} \label{h2}
|E|^{1/2} > \log(3+\delta^{-1}),
\end{equation}
we have
\begin{align} \label{mainesth3}
\||\lambda|^j |r_2 - r_1|\|_{L^p(\C)} \leq \theta_3 \bigg[|E|^{-m/2}\log(3+\delta^{-1})^{-(m-2)} + \frac{\delta}{|E|^{1/2p}}  e^{|E|(l+2)}\bigg],
\end{align}
for $j=-1,0,1$.
\end{prop}
\begin{proof}
Let choose $0<a_1 \leq 1 \leq a_2$ to be determined and let 
\begin{equation} \label{defdelta}
\delta = \|\Phi_2(E) - \Phi_1(E)\|_*.
\end{equation}
We split down the left term of \eqref{mainesth} as follows:
\begin{align*}
\| |\lambda|^j|r_2 - r_1|\|_{L^p(\C)} &\leq \||\lambda|^j|r_2 - r_1|\|_{L^p(|\lambda| <a_1)} +\| |\lambda|^j|r_2 - r_1|\|_{L^p(a_1<|\lambda| <a_2)}\\
&\quad + \| |\lambda|^j|r_2 - r_1|\|_{L^p(|\lambda| >a_2)}.
\end{align*}
From \eqref{estlem21} and \eqref{estlem22} we have
\begin{align}\label{esta}
\| |\lambda|^j |r_2 - r_1|\|_{L^p(|\lambda| <a_1)} &\leq c(N,p,m)  |E|^{-m/2} a_1^{m-1+j+2/p}, \\ \label{estb}
\||\lambda|^j |r_2 - r_1|\|_{L^p(|\lambda|>a_2)} &\leq c(N,p,m)  |E|^{-m/2} a_2^{-m-1+j+2/p}.
\end{align}
From Lemma \ref{lemdifh} and \eqref{defdelta} we obtain, for $j=-1,0,1$,
\begin{equation} \label{estc}
\| |\lambda|^j |r_2 - r_1|\|_{L^p(a_1<|\lambda| <a_2)} \leq c(D,N,p)\frac{\delta}{|E|^{1/2p}} \left(e^{(\sqrt{|E|}l+2)/a_1}+ e^{(\sqrt{|E|}l+2)a_2}\right).
\end{equation}

We now prove \eqref{mainesth}. Fix an energy $E < 0$ satisfying the hypothesis and define
\begin{equation}
a_2 = \frac{ 1}{ a_1} = \beta \log(3+\delta^{-1}),
\end{equation}
for $0<\beta< 1/(l\sqrt{|E|}+2)$. We choose $\delta_{\beta}(E) < 1$ such that for every $\delta \leq \delta_{\beta}(E)$, $a_2 >1$ (and so $a_1< 1$). Note that since $E_1 > (2R)^2$, the estimates in Lemma \ref{lemestr} hold for $a_1<1$ and $a_2 >1$.

The aim is to have \eqref{esta}, \eqref{estb} of the order $\log(3+\delta^{-1})^{-(m-2)}$. Indeed we have, for every $p\geq 1$ and $\delta \leq \delta_{\beta}(E)$,
\begin{align*}
a_1^{m-1+j+ 2/p} \leq c(\beta)\log(3+\delta^{-1})^{-(m-2)}, \quad a_2^{-m-1+j+2/p} \leq c(\beta) \log(3+\delta^{-1})^{-(m-2)},
\end{align*}
for $j=-1,0,1$.
Thus, for $\delta \leq \delta_{\beta}(E)$,
\begin{align*}
\| |\lambda|^j |r_2 - r_1|\|_{L^p(\C)} &\leq c(D,N,m,p,\beta)\bigg[|E|^{-m/2}\log(3+\delta^{-1})^{-(m-2)} \\
&\quad + \frac{\delta}{|E|^{1/2p}}  (3+\delta^{-1})^{\beta (\sqrt{|E|}l+2) }\bigg].
\end{align*}
Since by construction $\beta (\sqrt{|E|}l+2) <1$, we have that
\begin{align} \label{estexp}
\frac{\delta}{|E|^{1/2p}}  (3+\delta^{-1})^{\beta (\sqrt{|E|}l+2) } \to 0 \quad \text{for } \delta \to 0
\end{align}
more rapidly than the other term, at fixed $E$. This gives
\begin{equation} \label{estrdir}
\||\lambda|^j| r_2 - r_1\|_{L^p(\C)} \leq c(E,D,N,m,p,\beta)\left(\log(3+\delta^{-1})\right)^{-(m-2)},
\end{equation}
for $\delta \leq \tilde \delta_{\beta}(E)$ (where $\tilde \delta_{\beta}(E)$ is sufficiently small in order to estimate the term in \eqref{estexp}).
Estimate \eqref{estrdir} for general $\delta$ (with modified constant) follows from \eqref{estrdir} for $\delta \leq \tilde \delta_{\beta}(E)$ and the fact that $\||\lambda|^k |r_j|\|_{L^{p}(D)} < c(D,N,p)$, for $j=1,2$, $k=-1,0,1$ and $p \geq 1$: this follows from Lemma \ref{lemestr} (using the fact that $|E| > R$): indeed the estimate of Lemma \ref{lemb} hold for every $\lambda \in \C$, since $|E| > R$.\smallskip

In order to prove \eqref{mainesth2} we define, in \eqref{esta}-\eqref{estc},
\begin{equation}
a_2 = \frac{ 1}{ a_1} = 1 +\frac{\kappa \log(3+\delta^{-1})}{|E|^{1/2}},
\end{equation}
for any $0<\kappa< \frac{1}{l+2}$. Note that we have $a_2 > 1$ and $a_1 < 1$. Thus we find, for every $p\geq 1$, $j=-1,0,1$,
\begin{align*}
a_1^{m-1+j+ 2/p} \leq \frac{|E|^{(m-2)/2}}{\left(|E|^{1/2}+\kappa \log(3+\delta^{-1})\right)^{m-2}},\\
a_2^{-m-1+j+2/p} \leq \frac{|E|^{(m-2)/2}}{\left(|E|^{1/2}+\kappa \log(3+\delta^{-1})\right)^{m-2}}.
\end{align*}
We have also that
\begin{align*}
e^{(\sqrt{|E|}l+2)/a_1}+e^{(\sqrt{|E|}l+2)a_2} &\leq 2 e^{(l+2)\left(|E|^{1/2} + \kappa \log(3+\delta^{-1})\right)}\\
&= 2 (3+\delta^{-1})^{\kappa(l+2)}e^{(l+2)|E|^{1/2}}.
\end{align*}
Repeating the same arguments as above we obtain, for $\delta > 0$,
\begin{align*}
\| |\lambda|^j|r_2 - r_1|\|_{L^p(\C)} &\leq c(D,N,m,p)\bigg[|E|^{-1}\left(|E|^{1/2}+\kappa \log(3+\delta^{-1})\right)^{-(m-2)} \\
&\quad + \frac{\delta(3+\delta^{-1})^{\kappa(l+2)}}{|E|^{1/2p}} e^{(l+2)|E|^{1/2}}\bigg],
\end{align*}
which proves estimate \eqref{mainesth2}.\smallskip

We pass to estimate \eqref{mainesth3}. Take, in \eqref{esta}-\eqref{estc},
\begin{equation}
a_2 = \frac{ 1}{ a_1} = \log(3+\delta^{-1}).
\end{equation}
Define $\tilde \delta < 1$ such that for $\delta \leq \tilde \delta$ we have $a_2 > 1$ (so $a_1 < 1$).
From our assumption \eqref{h2} we have that $e^{(\sqrt{|E|}l+2)/a_1}+e^{(\sqrt{|E|}l+2)a_2} < 2 e^{|E|(l+2)}$. Then we obtain, using the same arguments as above,
\begin{align*}
 \||\lambda|^j |r_2 - r_1|\|_{L^p(\C)} \leq c(D,N,m,p) \bigg[|E|^{-m/2}\log(3+\delta^{-1})^{-(m-2)} + \frac{\delta}{|E|^{1/2p}}  e^{|E|(l+2)}\bigg],
\end{align*}
for $\delta \leq \tilde \delta$. To remove this last assumption we argue as for \eqref{mainesth}. This completes the proof of Proposition \ref{proprdir}.
\end{proof}

\section{Proof of Theorem \ref{maintheo}} \label{secrv}
We begin with a lemma which generalises \cite[Proposition 4.2]{S} to negative energy.

\begin{lem} \label{lemestmu}
Let $E < 0$ be such that $|E|\geq  E_1$, where $E_1$ is defined in Proposition \ref{proprdir}; let $v_1, v_2$ be two potentials satisfying \eqref{cv1}, \eqref{cv2}, with $\|v_j\|_{m,1} \leq N$, $\mu_1(z,\lambda), \mu_2(z,\lambda)$ the corresponding Faddeev functions and $r_1, r_2$ as defined in \eqref{defr}, \eqref{defr2}. Let $1 < s < 2$, and $\tilde s$ such that $1/\tilde s = 1/s - 1/2$. Then
\begin{align} \label{estmula1}
&\sup_{z \in \C} \|\mu_2(z,\cdot) - \mu_1(z,\cdot) \|_{L^{\tilde s}(\C)} \leq c(D,N,s,m) \| r_2 -r_1\|_{L^s(\C)},\\ \label{estmula2}
&\sup_{z \in \C} \left\|\frac{\partial \mu_2(z,\cdot)}{\partial \bar z} - \frac{\partial \mu_1(z,\cdot)}{\partial \bar z} \right\|_{L^{\tilde s}(\C)}\! \! \! \! \! \! \leq c(D,N,s,m)\Bigg[  \| r_2 -r_1\|_{L^s(\C)} \\ \nonumber
&\qquad+ |E|^{1/2}\left(\left\| \left(|\lambda| + \frac{1}{|\lambda|}\right)|r_2 - r_1|\right\|_{L^s(\C)} + \| r_2 -r_1\|_{L^s(\C)}  \right) \Bigg].
\end{align}
\end{lem}
\begin{proof}
We begin with the proof of \eqref{estmula1}. Let 
\begin{align} \label{defnu}
\nu(z,\lambda) &= \mu_2(z,\lambda) - \mu_1(z,\lambda).
\end{align}

From the $\bar \partial$-equation \eqref{dbar} we deduce that $\nu$ satisfies the following non-homogeneous $\bar \partial$-equation:
\begin{align} \label{dbar1}
\frac{\partial}{\partial \bar \lambda}\nu(z,\lambda) = r_1(z,\lambda)\overline{\nu(z,\lambda)}+ (r_2(z,\lambda) - r_1(z,\lambda))\overline{\mu_2(z,\lambda)},
\end{align}
for $\lambda \in \C$. Note that from Lemma \ref{lem21} we have that $\nu(z,\cdot) \in L^{\tilde s}(\C)$ for every $2 < \tilde s \leq \infty$. In addition, from Lemma \ref{lemb} (using the fact that $|E| > R$), we have that $\|r_j\|_{L^{p}(D)} < c(D,N,p,m)$, for $ 1 < p < \infty$, $j=1,2$. Then it is possible to use Lemma \ref{lemtech} in order to obtain
\begin{align*}
\|\nu(z,\cdot)\|_{L^{\tilde s}} &\leq c(D,N,s,m)\left\|\overline{\mu_2(z,\lambda)} (r_2(\lambda) - r_1(\lambda))\right\|_{L^s(\C)}\\
&\leq c(D,N,s,m) \sup_{z \in \C}\|\mu_2(z,\cdot)\|_{L^{\infty}} \left\| r_2 - r_1\right\|_{L^s(\C)}\\
&\leq c(D,N,s,m) \left\| r_2 - r_1\right\|_{L^s(\C)},
\end{align*}
and the constant is independent from $E$ for $|E| > R$, because of Lemma \ref{lem21} and Lemma \ref{lemb}.\smallskip

Now we pass to \eqref{estmula2}. To simplify notations we write, for $z,\lambda \in \C$,
\begin{equation} \nonumber
\mu_z^j(z,\lambda) = \frac{\partial \mu_j(z,\lambda)}{\partial z}, \quad \mu_{\bar z}^j(z,\lambda) = \frac{\partial \mu_j(z,\lambda)}{\partial \bar z},\quad j=1,2.
\end{equation}

From the $\bar \partial$-equation \eqref{dbar} we have that $\mu_z^j$ and $\mu_{\bar z}^j$ satisfy the following system of non-homogeneous $\bar \partial$-equations, for $j=1,2$:
\begin{align*}
\frac{\partial}{\partial \bar \lambda}\mu_z^j(z,\lambda) = r_j(z,\lambda)\left(\overline{\mu_{\bar z}^j(z,\lambda)}+\frac{\sqrt{|E|}}{2}\left(\bar \lambda -\frac 1 \lambda\right) \overline{\mu_j(z,\lambda)}\right),\\
\frac{\partial}{\partial \bar \lambda}\mu_{\bar z}^j(z,\lambda) = r_j(z,\lambda)\left(\overline{\mu_{ z}^j(z,\lambda)}+\frac{\sqrt{|E|}}{2}\left(\frac{1}{ \bar \lambda} - \lambda\right) \overline{\mu_j(z,\lambda)}\right).
\end{align*}
Define now $\mu_{\pm}^j(z,\lambda) = \mu_z^j(z,\lambda) \pm \mu_{\bar z}^j(z,\lambda)$, for $j=1,2$. 
Then they satisfy the following two non-homogeneous $\bar \partial$-equations:
\begin{equation} \nonumber
\frac{\partial}{\partial \bar \lambda}\mu_{\pm}^j(z,\lambda)= r_j(z,\lambda) \left( \pm \overline{\mu_{\pm}^j(z,\lambda)}+\frac{\sqrt{|E|}}{2}\left(  \left(\bar \lambda -\frac 1 \lambda\right)\pm \left(\frac{1}{ \bar \lambda} - \lambda\right)\right) \overline{\mu_j(z,\lambda)}\right).
\end{equation}
Finally define $\tau_{\pm}(z,\lambda) = \mu_{\pm}^2(z,\lambda)- \mu_{\pm}^1(z,\lambda)$. They satisfy the two non-homogeneous $\bar \partial$-equations below:
\begin{align*}
&\frac{\partial}{\partial \bar \lambda}\tau_{\pm}(z,\lambda) = \bigg[ \pm \left( r_1(z,\lambda)\overline{\tau_{\pm}(z,\lambda)}+ (r_2(z,\lambda)-r_1(z,\lambda))\overline{\mu_{\pm}^2(z,\lambda)} \right) \\ \nonumber
&\quad + \frac{\sqrt{|E|}}{2}\left(  \left(\bar \lambda -\frac 1 \lambda\right)\pm \left(\frac{1}{ \bar \lambda} - \lambda\right)\right) \left((r_2(z,\lambda)-r_1(z,\lambda))\overline{\mu_2(z,\lambda)}+ r_1(z,\lambda)\overline{\nu(z,\lambda)}\right) \bigg],
\end{align*}
where $\nu(z,\lambda)$ was defined in \eqref{defnu}.

Now remark that by Lemma \ref{lem21} and regularity assumptions on the potentials we have that $\mu_z^j(z,\cdot), \mu_{\bar z}^j(z,\cdot) \in L^{\tilde s}(\C) \cap L^{\infty}(\C)$ for any $\tilde s > 2$, $j=1,2$ (and their norms are bounded by a constant $C(D,N,p,m)$ thanks to Lemma \ref{lem21}). This, in particular, yields $\tau_{\pm}(z,\cdot) \in L^{\tilde s}(\C)$. These arguments, along with the above remarks on the $L^p$ boundedness of $r_j$, make possible to use Lemma \ref{lemtech}, which gives
\begin{align*}
&\|\tau_{\pm}(z,\cdot)\|_{L^{\tilde s}(\C)}\\
&\quad\leq c(D,N,s,m)\Bigg[\|r_2-r_1 \|_{L^s(\C)}\|\mu_{\pm}^2(z,\cdot)\|_{L^{\infty}(\C)}\\
&\qquad +\sqrt{|E|}\Bigg(\left\| \left(|\lambda| + \frac{1}{|\lambda|}\right)|r_2 - r_1|\right\|_{L^s(\C)}\|\mu_2(z,\cdot)\|_{L^{\infty}(\C)}\\
&\qquad + \| (|\lambda|+|\lambda|^{-1})r_1\|_{L^2(\C)}\|\nu(z,\cdot)\|_{L^{\tilde s}(\C)} \Bigg)\Bigg]\\ 
&\quad \leq c(D,N,s,m)\Bigg[\|r_2-r_1\|_{L^s(\C)}+\sqrt{|E|}\Bigg(\left\| \left(|\lambda| + \frac{1}{|\lambda|}\right)|r_2 - r_1|\right\|_{L^s(\C)}\\
&\qquad + \||\lambda||r_2-r_1|\|_{L^s(\C)} \Bigg)\Bigg],
\end{align*}
where we used H\"older's inequality \eqref{holder} (since $1/s = 1/2 + 1/ \tilde s$) and estimate \eqref{estmula1}. Again, the constants are independent from $E$ since $|E| > R$.

The proof of \eqref{estmula2} now follows from this last inequality and the fact that $\mu^2_{\bar z}-\mu^1_{\bar z} = \frac{1}{2}(\tau_+ -\tau_-)$.
\end{proof}
\begin{rem}
We also have proved that
\begin{align*}
&\sup_{z \in \C} \left\|\frac{\partial \mu_2(z,\cdot)}{\partial z} - \frac{\partial \mu_1(z,\cdot)}{\partial z} \right\|_{L^{\tilde s}(\C)}\! \! \! \! \! \! \leq c(D,N,s,m)\Bigg[  \| r_2 -r_1\|_{L^s(\C)} \\ \nonumber
&\qquad+ |E|^{1/2}\left(\left\| \left(|\lambda| + \frac{1}{|\lambda|}\right)|r_2 - r_1|\right\|_{L^s(\C)} + \| r_2 -r_1\|_{L^s(\C)}  \right) \Bigg].
\end{align*}
\end{rem}

\begin{proof}[Proof of Theorem \ref{maintheo}]
We recall the derivation of an explicit formula for the potential, taken from \cite{N3}.

Let $v(z)$ be a potential which satisfies the hypothesis of Theorem \ref{maintheo} and $\mu(z,\lambda)$ the corresponding Faddeev functions. Since $\mu(z,\lambda)$ satisfies the estimates of Lemma \ref{lem21}, the $\bar \partial$-equation \eqref{dbar} and $b(\lambda,E)$ decreases at infinity like in Lemma \ref{lemb}, it is possible to write the following development:
\begin{equation} \label{devel}
\mu(z,\lambda) = 1 + \frac{\mu_{-1}(z)}{\lambda}+O\left(\frac{1}{|\lambda|^2}\right), \qquad \lambda \to \infty,
\end{equation}
for some function $\mu_{-1}(z)$. If we insert \eqref{devel} into equation \eqref{equa}, for $\psi(z,\lambda) = e^{-\frac{\sqrt{|E|}}{2} \left(z/\lambda + \bar z \lambda\right)} \mu(z,\lambda)$, we obtain, letting $\lambda \to \infty$,
\begin{equation}
v(z) = -2 |E|^{1/2} \frac{\partial \mu_{-1}(z)}{\partial z}, \qquad z \in \C.
\end{equation}
More explicitly, we have, as a consequence of \eqref{dbar},
\begin{equation} \nonumber
\mu(z,\lambda)-1=\frac{1}{2\pi i}\int_{\C}\frac{r(z,\lambda')}{\lambda'-\lambda}\overline{\mu(z,\lambda')}d\lambda' \, d \bar \lambda'.
\end{equation}
By Lebesgue's dominated convergence (using Lemma \ref{lemb}) we obtain
\begin{equation}\nonumber
\mu_{-1}(z)=-\frac{1}{2\pi i}\int_{\C}r(z,\lambda)\overline{\mu(z,\lambda)}d\lambda \, d \bar \lambda,
\end{equation}
and the explicit formula
\begin{equation} \label{expl}
v(z) = \frac{|E|^{1/2}}{\pi i}\int_{\C}r(z,\lambda)\left(\frac{|E|^{1/2}}{2}\left( \bar \lambda -\frac 1 \lambda \right)\overline{\mu(z,\lambda)}+\overline{\left(\frac{\partial \mu(z,\lambda)}{\partial \bar z}\right)}\right)d\lambda\, d\bar \lambda.
\end{equation}
Formula \eqref{expl} for $v_1$ and $v_2$ yields
\begin{align*}
v_2(z)-v_1(z)&=\frac{|E|^{1/2}}{\pi i}\int_{\C}\Bigg[\frac{|E|^{1/2}}{2}\left( \bar \lambda -\frac 1 \lambda \right)\big((r_2 - r_1) \overline{\mu_2}+ r_1 (\overline{\mu_1}-\overline{\mu_2})\big)\\
&\quad +(r_2-r_1)\overline{\left(\frac{\partial \mu_2}{\partial \bar z}\right)}+r_1 \overline{\left(\frac{\partial \mu_2}{\partial \bar z}-\frac{\partial \mu_1}{\partial \bar z}\right)}\Bigg]d\lambda\, d\bar \lambda.
\end{align*}
Then, using several times H\"older's inequality \eqref{holder}, we find
\begin{align*}
|v_2(z)-v_1(z)|&\leq \frac{|E|^{1/2}}{\pi}\Bigg[\frac{|E|^{1/2}}{2}\Bigg(\left\|\left(\bar \lambda - \frac 1 \lambda \right)(r_2-r_1)\right\|_{L^1}\|\mu_2(z,\cdot)\|_{L^{\infty}}\\
&\quad+\left\|\left(\bar \lambda - \frac 1 \lambda \right)r_1\right\|_{L^{\tilde p'}}\|\mu_2(z,\cdot) - \mu_1(z,\cdot)\|_{L^{\tilde p}}\Bigg)\\
&\quad + \|r_2-r_1\|_{L^p}\left\|\frac{\partial \mu_2(z,\cdot)}{\partial \bar z} \right\|_{L^{p'}}\\
&\quad + \|r_1\|_{L^{\tilde p'}}\left\| \frac{\partial \mu_2(z,\cdot)}{\partial \bar z}-\frac{\partial \mu_1(z,\cdot)}{\partial \bar z}\right\|_{L^{\tilde p}}\Bigg],
\end{align*}
for $1<p<2$, $\tilde p$ such that $1/\tilde p = 1 / p - 1/2$ and $1/p + 1/p' = 1/{\tilde p} + 1/{\tilde p'} = 1$. From Lemmas \ref{lemestmu}, \ref{lem21} and \ref{lemestr} we obtain
\begin{align*}
|v_2(z)-v_1(z)|&\leq c(D,N,m,p)|E|^{1/2}\Bigg[|E|^{1/2}\Bigg(\left\|\left(\bar \lambda - \frac 1 \lambda \right)(r_2-r_1)\right\|_{L^1}\\
&\quad + \sum_{k=-1}^1 \| |\lambda|^k|r_2 - r_1|\|_{L^p(\C)}\Bigg)+\|r_2-r_1\|_{L^p}\Bigg].
\end{align*}
Now Proposition \ref{proprdir} gives
\begin{equation}
\|v_2 - v_1\|_{L^{\infty}(D)}\leq c(E,D,N,m)(\log(3+\|\Phi_2(E)-\Phi_1(E)\|_*^{-1}))^{-(m-2)},
\end{equation}
which is \eqref{est1}. Estimates \eqref{est2} and \eqref{est3} are also obtained as a consequence of the above inequality and Proposition \ref{proprdir}.
This finishes the proof of Theorem \ref{maintheo}.
\end{proof}

\end{document}